\documentclass[a4paper, 12pt]{article}
\pdfoutput=1

\usepackage[T1]{fontenc}
\usepackage[utf8]{inputenc}

\usepackage{amsfonts,amsmath,amssymb,mathtools,dsfont,microtype} 
\usepackage[dvipsnames]{xcolor} 
\usepackage[noBBpl]{mathpazo} 

\DeclareFontFamily{U} {cmr}{}

\DeclareFontShape{U}{cmr}{m}{n}{
	<-6> cmr5
	<6-7> cmr6
	<7-8> cmr7
	<8-9> cmr8
	<9-10> cmr9
	<10-12> cmr10
	<12-> cmr12}{}

\DeclareSymbolFont{Xcmr} {U} {cmr}{m}{n}
\DeclareMathSymbol{\Delta}{\mathord}{Xcmr}{'001}
\DeclareMathSymbol{\Upsilon}{\mathord}{Xcmr}{'007}
\DeclareMathSymbol{\Omega}{\mathord}{Xcmr}{'012}

\usepackage[labelsep = period, labelfont = bf, justification = centering]{caption}
\usepackage{float,graphicx,subcaption}
\DeclareCaptionSubType*[roman]{figure}

\usepackage{booktabs,makecell}
\usepackage{enumitem,mdwlist}
\setlist[itemize]{topsep=0ex,itemsep=0ex,parsep=0.4ex}
\setlist[enumerate]{topsep=0ex,itemsep=0ex,parsep=0.4ex}

\usepackage{pgf,tikz,ifthen,calc}
\usepackage{float, graphicx}
\usepackage{tkz-euclide}
\tkzSetUpPoint[fill=black, size = 3pt]
\tkzSetUpLine[color=black, line width=0.6pt]
\tkzSetUpCircle[color=black, line width=0.6pt]

\usepackage[left = 2.5cm, right = 2.5cm, top = 2.5cm, bottom = 2.5cm, headsep = 12pt, headheight = 15pt]{geometry}
\usepackage{parskip}

\usepackage[hyphens]{url} 
\usepackage[linktoc = all, hidelinks, colorlinks, unicode=true, pagebackref=true]{hyperref} 
\usepackage[capitalise, compress, nameinlink, noabbrev]{cleveref} 

\hypersetup{linkcolor={blue!70!black}, citecolor={green!70!black}, urlcolor={blue!70!black}}

\usepackage{amsthm,thmtools,thm-restate}
\declaretheorem[name = Theorem, numberwithin = section, style = plain]{thm}

\declaretheorem[name = Corollary, numberlike = thm, style = plain]{cor}
\declaretheorem[name = Conjecture, numberlike = thm, style = plain]{conj}
\declaretheorem[name = Definition, numberlike = thm, style = definition]{defi}
\declaretheorem[name = Lemma, numberlike = thm, style = plain]{lem}

\crefname{defi}{Definition}{Definitions}
\crefname{thm}{Theorem}{Theorems}
\crefname{lem}{Lemma}{Lemmas}
\crefname{conj}{Conjecture}{Conjectures}
\crefname{claim}{Claim}{Claims}
\crefname{cor}{Corollary}{Corollaries}
\crefname{obs}{Observation}{Observations}
\crefname{prop}{Proposition}{Propositions}
\crefname{que}{Question}{Questions}
\crefname{rem}{Remark}{Remarks}
\crefname{subsection}{\S}{\S\S}

\DeclareFontFamily{U}{matha}{\hyphenchar\font45}
\DeclareFontShape{U}{matha}{m}{n}{
	<5> <6> <7> <8> <9> <10> gen * matha
	<10.95> matha10 <12> <14.4> <17.28> <20.74> <24.88> matha12
}{}
\DeclareSymbolFont{matha}{U}{matha}{m}{n}

\DeclareMathSymbol{\specialuparrow}{\mathrel}{matha}{"D2}
\DeclareMathSymbol{\specialrightarrow}{\mathrel}{matha}{"D1}

\renewcommand*{\backref}[1]{}
\renewcommand*{\backrefalt}[4]{
	\ifcase #1 Not cited.%
	\or $\specialuparrow$#2%
	\else $\specialuparrow$#2%
	\fi%
}

\renewcommand{\epsilon}{\varepsilon}
\renewcommand{\ge}{\geqslant}
\renewcommand{\le}{\leqslant}
\renewcommand{\geq}{\geqslant}
\renewcommand{\leq}{\leqslant}

\renewcommand{\subset}{\subseteq}
\renewcommand{\supset}{\supseteq}

\DeclarePairedDelimiter{\abs}{\lvert}{\rvert}
\DeclarePairedDelimiter{\ceil}{\lceil}{\rceil}
\DeclarePairedDelimiter{\floor}{\lfloor}{\rfloor}
\DeclarePairedDelimiter{\set}{\lbrace}{\rbrace}
\DeclarePairedDelimiter{\norm}{\lVert}{\rVert}
\DeclarePairedDelimiter{\rou}{\lfloor}{\rceil}

\newcommand*{\bH}{\mathbb{H}}
\newcommand*{\bN}{\mathbb{N}}

\newcommand*{\bR}{\mathbb{R}}
\newcommand*{\bZ}{\mathbb{Z}}

\newcommand*{\cG}{\mathcal{G}}

\newcommand*{\cL}{\mathcal{L}}
\newcommand*{\cO}{\mathcal{O}}
\newcommand*{\cP}{\mathcal{P}}

\newcommand*{\cY}{\mathcal{Y}}

\DeclareMathOperator{\Cay}{Cay}
\DeclareMathOperator{\dist}{dist}
\DeclareMathOperator{\round}{round}

\newcommand{\dH}{d^{\bH}}
\newcommand{\defn}[1]{\textcolor{Maroon}{\emph{#1}}}
\newcommand*{\eps}{\varepsilon}

\title{Disproof of the tree product conjecture\\ via the Heisenberg group}
\author{Freddie Illingworth\footnotemark[2] \qquad Sergey Norin\footnotemark[3] \\
Raphael Steiner\footnotemark[4]}
\date{\today}

\begin{document}

\maketitle

\renewcommand{\thefootnote}{\fnsymbol{footnote}} 

\footnotetext[2]{Department of Mathematics, University College London, UK (\textsf{\href{mailto:f.illingworth@ucl.ac.uk}{f.illingworth@ucl.ac.uk}}).}

\footnotetext[3]{Department of Mathematics and Statistics, McGill University, Montr\'{e}al, Canada (\textsf{\href{mailto:sergey.norin@mcgill.ca}{sergey.norin\allowbreak@mcgill.ca}}). Research supported by NSERC.}

\footnotetext[4]{Department of Mathematics, ETH Z\"{u}rich, Switzerland (\textsf{\href{mailto:raphaelmario.steiner@math.ethz.ch}{raphaelmario.steiner@math.ethz.ch}}). Research supported by the SNSF Ambizione Grant No. 216071.}

\renewcommand{\thefootnote}{\arabic{footnote}} 

\begin{abstract}
    \noindent Product structure theory aims to understand complex graphs by embedding them into products of simpler graphs. In this direction, Campbell, Distel, Gollin, Harvey, Hendrey, Hickingbotham, Mohar and Wood (2022) put forth the conjecture that all graphs of degree-$d$ polynomial growth (i.e., where balls of radius $r$ have $\cO(r^d)$ vertices) can be embedded into the strong product of $d$ trees, each with linear growth, and a constant-size clique. In this paper, we disprove this conjecture for $d = 4$. The counterexamples are finite subgraphs of a Cayley graph of the discrete $3$-dimensional Heisenberg group $\bH(\bZ)$. These graphs were first proposed by Huang and McCarty as potential counterexamples to the conjecture. A key technical tool of our proof is the ``quantitative central collapse'' theorem due to Cheeger, Kleiner and Naor (2011), guaranteeing that every Lipschitz map from the continuous Heisenberg group $\bH$ to the function space $L_1$ collapses along a central line.
\end{abstract}

\section{Introduction}

Product structure theory aims to understand the structure of graphs in complicated classes better by representing them as subgraphs of products of graphs in simpler classes, typically with bounded tree-width. The area emerged in 2019, sparked in particular by the work of Dujmovi\'{c}, Joret, Micek, Morin, Ueckerdt and Wood~\cite{DujmovicJMMUW2020QueueNumber}, who proved their now famous \emph{Planar Graph Product Structure Theorem}, stating that every planar graph is isomorphic to a subgraph of the \emph{strong product}\footnote{Recall that given two graphs $H_1, H_2$, the \defn{strong product} \defn{$H_1\boxtimes H_2$} of $H_1, H_2$ is defined as the graph with vertex set $V(H_1)\times V(H_2)$, in which two distinct vertices $(u_1,u_2), (v_1,v_2)$ are adjacent if and only if for both $i\in \{1,2\}$ we have that $u_i=v_i$ or $u_iv_i\in E(H_i)$.} of a path and a graph with bounded tree-width. A consequence of this result solved a longstanding open problem about the queue-number of planar graphs, and many more applications followed soon after. Since this seminal result, the area has flourished into a highly active and important research branch of structural graph theory and has also found applications to other areas of combinatorics. See~\cite{A,B,C,D,E,F,G,H,I} for only a small selection of important recent results from graph product structure theory, and~\cite{J} for an example of an application to Ramsey theory.

In this work, we shall be concerned with product structure theory for classes of graphs with \defn{polynomial growth}. To make this precise, let us define, given a (locally finite, possibly infinite) graph $G$, its \defn{growth function} $f_G \colon \bN \to \bN\cup\{\infty\}$ as the function that assigns to each $r\in \bN$ the supremum of the sizes of balls of radius $r$ in $G$. We then say that a graph class $\cG$ has \defn{degree-$d$ polynomial growth} if there exists an absolute constant $c>0$ such that $f_G(r) \le cr^d$ for every $G \in \cG$ and $r\in \bN$. 

Graph classes with polynomial growth form a fundamental topic of research in group theory, see e.g.~\cite{G1,G2,G3,G4,G5,G6,G7} (here, the focus is on Cayley graphs, in particular). However, they also play an important role in a variety of further areas such as in metric geometry~\cite{M1}, in algebraic graph theory~\cite{A1,A2,A3,A4,A5,A6} and in the study of random infinite planar graphs~\cite{R1,R2}.

Graph classes with \defn{linear growth} (i.e., $d = 1$) were studied by Campbell, Distel, Gollin, Harvey, Hendrey, Hickingbotham, Mohar, and Wood~\cite{MR4614535} who proved that they must have bounded tree-width. More strongly, they showed that graphs in classes of linear growth are isomorphic to subgraphs of strong products $T\boxtimes K_C$, where $T$ is a bounded-degree tree and $C$ is a constant depending only on the graph class. Another relevant product structure result comes from the seminal work of Krauthgamer and Lee~\cite{M1}, who showed that graphs in classes with degree-$d$ polynomial growth are subgraphs of the strong product of $\cO(d\log d)$ paths, and that the $\cO(d\log d)$ bound in this result is best-possible.

Pause to note that if graph classes $\cG_1,\dots,\cG_k$ are such that $\cG_i$ has degree-$d_i$ polynomial growth, then the class of graphs of the form $G_1 \boxtimes \dots \boxtimes G_k$ with $G_i\in \cG_i$ has degree-$(d_1+\dots+d_k)$ polynomial growth. 

With this in mind, it is natural to ask whether the number of factors $\cO(d\log d)$ in Krauthgamer and Lee's result can be improved to the tight number $d$ when allowing the factors of the product to be linear-growth graphs of slightly more general shape than just paths (which would then recover precisely the same degree of polynomial growth in the products as for the original graph class). In this direction, Campbell et al.~put forth the following conjecture, which would not only give a very satisfactory affirmative answer to this question but also strengthen their aforementioned result on graph classes of linear growth.

\begin{conj}[Campbell, Distel, Gollin, Harvey, Hendrey, Hickingbotham, Mohar, Wood~{\cite[Conj.~15]{MR4614535}}]\label{con:wood}
    There exist functions $g \colon \bR \times \bN \to \bN$ and $h \colon \bR \times \bN \to \bN$ such that for every $c\in \bR_{\ge 1}$ and $d \in \bN$, every finite graph $H$ with growth $f_H(r)\le cr^d$ is isomorphic to a subgraph of $T_1 \boxtimes \dots \boxtimes T_d \boxtimes K_{g(c,d)}$, where each $T_i$ is a tree with growth $f_{T_i}(r)\le h(c,d)r$. 
\end{conj}

As the main result of this paper, we show that this beautiful conjecture, unfortunately, is false, already when $d=4$.

Concretely, we prove the following result about a Cayley graph of the discrete Heisenberg group (see \cref{sec:Heisenberg} for definitions). While this is an infinite graph, a relatively standard compactness argument (given in \cref{sec:compactness}) then shows that \cref{con:wood} is false also for finite graphs.

\begin{thm}\label{thm:main}
    There is a Cayley graph of the discrete Heisenberg group, $\Cay(\bH(\bZ), T)$, whose growth function is at most $27\cdot r^4$ and such that, for every $C \in \bN$ and every quadruple of finite or countably infinite trees $F_1, F_2, F_3, F_4$ satisfying $f_{F_i}(r) \leq Cr$ \textup{(}$i = 1, 2, 3, 4$\textup{)}, no subgraph of $F_1\boxtimes F_2\boxtimes F_3 \boxtimes F_4 \boxtimes K_C$ is isomorphic to $\Cay(\bH(\bZ), T)$.
\end{thm}

\paragraph*{Overview and organisation.} In the remainder of the paper, we present our proof of \cref{thm:main} as well as a finite version (\cref{thm:main-finite}), the latter then disproves \cref{con:wood}. Our proof is based on a deep result of Cheeger, Kleiner and Naor~\cite{CKN11}, which guarantees that every Lipschitz map from the continuous $3$-dimensional Heisenberg group to the function space $L_1$ collapses along a so-called central line. Several technical hurdles need to be overcome to use this result for the purpose of proving our graph-theoretical result (\cref{thm:main}), and we address these in different sections.

We start in \cref{sec:Heisenberg} by formally introducing the continuous and discrete Heisenberg groups, record some relevant metric and measure-theoretic properties of these groups, and explicitly describe the Cayley graph used in the proof of \cref{thm:main}. In \cref{sec:collapse} we formally state the key ingredient from the work of Cheeger, Kleiner and Naor~\cite{CKN11}, called the ``quantitative central collapse theorem''. In \cref{sec:discrete} we then convert this result, which in its original form speaks about properties of $L_1$-embeddings of the \emph{continuous} Heisenberg group, into a ready-to-use consequence for the \emph{discrete} Heisenberg group. Finally, in \cref{sec:isometry} we record a simple additional ingredient about isometric $L_1$-embeddings of infinite trees and then go on in \cref{sec:proof} to use this fact and the discrete version of Cheeger, Kleiner and Naor's result to prove our main result, \cref{thm:main}. Finally, in \cref{sec:compactness} we supply the aforementioned compactness argument, which allows us to convert \cref{thm:main} to a finite version, \cref{thm:main-finite}, which then disproves \cref{con:wood}.

\paragraph*{Notation and Terminology.} In this paper, some considered graphs will be finite while some will be infinite. When referring to ``graph'' without any additional specification, we will assume a finite graph by default. Given a (possibly infinite) graph $G$, we denote by $V(G)$ its vertex-set and by $E(G)$ its edge-set. Throughout, $\mathbb{N}$ denotes the set of natural numbers excluding $0$. Given real positive expressions $a,b$, we use the notation \defn{$a \gtrsim b$} to express that $\frac{a}{b}$ is bounded from below by an implicit, absolute, positive constant. We also use the notation \defn{$a\asymp b$} to express that both $a\gtrsim b$ and $b\gtrsim a$.
For a real number $a$, we denote by \defn{$\rou{a}$} the closest integer to $a$ (defined as $a-1/2$ if $a$ is half-integral).

\paragraph*{Acknowledgements.} This research was carried out at the workshop ``Global Structure and Geometry of Graphs'', at the MATRIX Research Institute in Australia. We thank the Institute for its hospitality and the workshop organisers and participants for creating a stimulating working environment. We would like to thank Chun-Hung Liu, David Wood, and Jung-Hon Yip for fruitful discussions on the topic of this paper during the workshop. Finally, it is important to note that the idea of investigating \cref{con:wood} for the example of Cayley graphs of the Heisenberg group is not ours, and was previously proposed by Owen Huang and Rose McCarty (personal communication) and came about when Huang was studying geometric group theory. The examples were also explicitly proposed for investigation by McCarty at the \href{https://web.math.princeton.edu/~tunghn/2024bbd.pdf}{2024 Barbados Graph Theory workshop}. Our main contribution here is to settle this problem.

\paragraph*{AI Disclosure.} We used ChatGPT 5.5 Pro to assist us with working out the details of the compactness argument presented in \cref{sec:compactness}. Its output was carefully checked and rewritten by the authors. The content of all other sections is due entirely to the authors.

\section{The Heisenberg group}\label{sec:Heisenberg}

We give a brief overview of the Heisenberg group and its properties that we require. See, for example, \cite[\S\S1.1, 2.3, 2.4]{CKN11} for further details.

The \defn{continuous Heisenberg group}, \defn{$\bH$}, consists of $3 \times 3$ upper triangular real matrices with ones on the diagonal, under the operation of matrix multiplication. Identifying the triple $(a, b, c) \in \bR^3$ with the matrix
\begin{equation*}
    \begin{pmatrix}
        1 & a & \tfrac{c + ab}{2} \\
        0 & 1 & b \\
        0 & 0 & 1
    \end{pmatrix}
\end{equation*}
gives the view of $\bH$ as $\bR^3$ equipped with the non-commutative group operation
\begin{equation}\label{eq:product}
    (a, b, c) \cdot (a', b', c') = (a + a', b + b', c + c' + ab' - ba')
\end{equation}
and identity element $(0, 0, 0)$.
The group has an associated \defn{Carnot-Carath\'{e}odory metric}, \defn{$\dH$}. We will not need to define this metric here but will need the following two equations (cf.~\cite[(2.16)]{CKN11} and the sentence directly after that equation):
\begin{align}
    \dH((a, b, c), (a', b', c')) & \asymp \sqrt{(a - a')^2 + (b - b')^2} + \abs{c - c' + ab' - ba'}^{1/2}, \label{eq:dH} \\
    \dH((a, b, c), (a, b, c')) & = \abs{c - c'}^{1/2}. \label{eq:dHv}
\end{align}
$\bH$ has an important automorphism, \defn{$A_R$}$\colon \bR^3 \to \bR^3, (a, b, c) \mapsto (Ra, Rb, R^2 c)$, for $R > 0$, which is an $R$-homothety of $(\bH, \dH)$: $\dH(A_Rx, A_Ry) = R \cdot \dH(x, y)$ (cf.~\cite[(1.2)]{CKN11}).

The \defn{centre} of $\bH$, \defn{$Z(\bH)$}, is the normal subgroup of $\bH$ consisting of all elements that commute with the entire group. This is $\set{0} \times \set{0} \times \bR$.
Note that cosets of $Z(\bH)$ are parametrised by $\bR^2$: $(a, b, c) \cdot Z(\bH) = \set{a} \times \set{b} \times \bR$. 
It follows that $A_R(x \cdot Z(\bH)) = A_R x \cdot Z(\bH)$ for every $x \in \bH$ and $R>0$.

The \defn{discrete Heisenberg group}, \defn{$\bH(\bZ)$}, is the integer lattice $\bZ^3$ equipped with the product~\eqref{eq:product} (one can easily check that indeed, inverses of integer points of $\mathbb{H}$ again have integer coordinates, so this indeed forms a group). It is furthermore easy to check that $\mathbb{H}(\mathbb{Z})$ is generated by $T \coloneqq \set{(\pm 1, 0, 0), (0, \pm 1, 0), \allowbreak (0, 0, \pm 1)}$. This generating set yields the Cayley graph which we use in \cref{thm:main}.

\begin{defi}[Cayley graph of discrete Heisenberg group]
    The Cayley graph of $\bH(\bZ)$ with respect to generating set $T$, denoted \defn{$\Cay(\bH(\bZ), T)$}, has vertex-set $\bZ^3$ and the six neighbours of each $(a, b, c) \in \bZ^3$ are given by
    \begin{align*}
        (a, b, c) \cdot (\pm 1, 0, 0) & = (a \pm 1, b, c \mp b),\\
        (a, b, c) \cdot (0, \pm 1, 0) & = (a, b \pm 1, c \pm a),\\
        (a, b, c) \cdot (0, 0, \pm 1) & = (a, b, c \pm 1).
    \end{align*}
\end{defi}

We will need the following simple upper bound on the growth of balls in $\Cay(\bH(\bZ), T)$.

\begin{lem}\label{lem:growth}
    For every $r\in \bN$, the growth function satisfies $f_{\Cay(\bH(\bZ), T)}(r) \leq 27\cdot r^4$.
\end{lem}

\begin{proof}
    By definition of the growth function, it suffices to show that every ball of radius $r$ around any vertex of $\Cay(\bH(\bZ), T)$ contains at most $27\cdot r^4$ vertices. Since $\Cay(\bH(\bZ), T)$ is a Cayley graph and thus vertex-transitive, it suffices to show this fact for the vertex $(0, 0, 0)$. We claim that every vertex $(a, b, c)\in V(\Cay(\bH(\bZ), T))$ at distance at most $r$ from $(0, 0, 0)$ satisfies $\abs{a} \leq r, \abs{b} \leq r$ and $\abs{c} \leq r^2$. Since $(2r+1)(2r+1)(2r^2+1)\le 27\cdot r^4$, this will imply the assertion of the lemma.

    First note that any move along an edge of $\Cay(\bH(\bZ), T)$ does not change the values of the first two coordinates by more than $1$. This immediately implies the bounds on $a$ and $b$. It then follows that any move along an edge between two vertices \emph{inside} the radius-$r$-ball around $(0, 0, 0)$ changes the value of the third coordinate by at most $r$. This implies the bound on $c$.
\end{proof}

The graph $\Cay(\bH(\bZ), T)$ induces a metric on $\bH(\bZ)$ given by the distance in the graph. This metric is precisely the \defn{word metric}, \defn{$d_T$}, on $\bH(\bZ)$: $d_T(x, y)$ is the length of the shortest word in the elements of $T$ that is equal to $x^{-1} y$. The word metric is left-invariant and the spaces $(\bH(\bZ), \dH)$ and $(\bH(\bZ), d_T)$ are bi-Lipschitz equivalent (cf.~\cite[p.~296]{CKN11}).

\subsection{Haar measure on the continuous Heisenberg group}\label{subsec:Haar}

The quantitative central collapse theorem of Cheeger, Kleiner, and Naor~\cite[Thm.~1.1]{CKN11} requires the notion of measures on the Heisenberg group as well as on specific subsets of the group.
The Haar measure, \defn{$\mu$}, on $\bH$ coincides with the standard 3-dimensional Lebesgue measure, $\cL_3$.

\begin{lem}\label{lem:Haar-ball}
    For all $z \in \bH$ and $r>0$,
    \begin{equation*}
        \mu(B_{\dH}(z, r)) \asymp r^4.
    \end{equation*}
\end{lem}

\begin{proof}
    For reals $x$ and $y$, we use \defn{$[x \pm y]$} to denote the interval $[x - y, x + y]$. Let $z = (a, b, c)$. 
    Equation~\eqref{eq:dH} implies that there is an absolute constant $C > 0$ such that
    \begin{equation}\label{eq:ball}
        \begin{aligned}
            & \bigcup_{\substack{a' \in [a \pm r/C], \\ b' \in [b \pm r/C]}} \set{a'} \times \set{b'} \times [c + ab' - ba' \pm r^2/C] \subset B_{\dH}(z, r) \\
            \subseteq & \bigcup_{\substack{a' \in [a \pm C r], \\ b' \in [b \pm C r]}} \set{a'} \times \set{b'} \times [c + ab' - ba' \pm C r^2].
        \end{aligned}
    \end{equation}
    The first union has Lebesgue measure $(2r/C)^2 \times (2r^2/C) = 8r^4/C^3$ and the second has Lebesgue measure $8 C^3 r^4$. This implies the claimed result.
\end{proof}

On cosets of the centre of $\bH$ (viewed as the $z$-axis
in $\bR^3$), let $\mu$ be the 1-dimensional Lebesgue measure, $\cL_1$. While this is an abuse of notation, it will always be clear from context to which measure $\mu$ is referring.

\begin{lem}\label{lem:Haar-coset}
    For all $x \in \bH$ and $r > 0$,
    \begin{equation*}
        \mu(x \cdot Z(\bH) \cap B_{\dH}(x, r)) = 2r^2.
    \end{equation*}
\end{lem}

\begin{proof}
    Let $x = (a, b, c)$ and so $x \cdot Z(\bH) = \set{a} \times \set{b} \times \bR$. By \eqref{eq:dHv},
    \begin{equation*}
        x \cdot Z(\bH) \cap B_{\dH}(x, r) = \set{a} \times \set{b} \times [c - r^2, c + r^2]
    \end{equation*}
    and so $\mu(x \cdot Z(\bH) \cap B_{\dH}(x, r)) = \cL_1([c - r^2, c + r^2]) = 2r^2$.
\end{proof}

On pairs $(x, y)$ that lie in the same coset of the centre (that is, $x^{-1} y \in \bZ(\bH)$), the corresponding measure is the 2-dimensional Lebesgue measure, $\cL_2$.
Recall that the cosets of the centre ($\set{a} \times \set{b} \times \bR$, for $a, b \in \bR$) are parametrised by $\bR^2$. Let $F$ be a set of pairs $(x, y)$ such that $x$ and $y$ lie in the same coset of the centre. In particular, $F$ is partitioned into $F_{a, b} \coloneqq \set{(x, y) \in F \colon x, y \in \set{a} \times \set{b} \times \bR}$ ($a, b \in \bR$). We define a measure on such subsets by
\begin{equation*}
    \mu(F) \coloneqq \int_{\bR^2} \cL_2(F_{a, b}) \ \textrm{d} \cL_2(a, b).
\end{equation*}

\section{Quantitative central collapse}\label{sec:collapse}

With the relevant notions set in place, we can now formally state the aforementioned theorem of Cheeger, Kleiner, and Naor~\cite[Thm.~1.1]{CKN11}.

\begin{thm}[Quantitative central collapse]\label{thm:QCC}
    There is a universal constant $\delta \in (0, 1)$ such that for every $p \in \bH$, every $1$-Lipschitz $f \colon B_{\dH}(p, 1) \to L_1$, and every $\eps \in (0, 1/4)$, there is $r \geq \eps/2$ such that, with respect to the measure $\mu$ defined in \cref{subsec:Haar}, for at least half of the points $x \in B_{\dH}(p, 1/2)$, at least half of 
    \begin{equation*}
        \set[\big]{(x_1, x_2) \in B_{\dH}(x, r) \times B_{\dH}(x, r) \colon x_1^{-1}x_2 \in Z(\bH), \, \dH(x_1, x_2) \in \bigl[\tfrac{1}{2}\eps r, \tfrac{3}{2} \eps r\bigr]}
    \end{equation*}
    satisfy
    \begin{equation*}
        \frac{\norm{f(x_1) - f(x_2)}_{L_1}}{\dH(x_1, x_2)} \leq \bigl(\log(1/\eps)\bigr)^{-\delta}.
    \end{equation*}
\end{thm}

We now deduce a slightly weaker form of \cref{thm:QCC} that will be more convenient to apply and sufficient for our purposes.

\begin{cor}\label{cor:quant-centr-collapse}
    For all $\gamma > 0$ there is $\eps_0 > 0$ such that for every $1$-Lipschitz $f \colon (\bH, \dH) \to L_1$ there is $x \in \bH$ and $r \geq \eps_0$ such that, with respect to the measure $\mu$ defined in \cref{subsec:Haar}, an absolute positive constant proportion of $y \in x \cdot Z(\bH) \cap B_{d^\bH}(x, r)$ satisfy
    \begin{equation*}
        \norm{f(x) - f(y)}_{L_1} \leq \gamma r.
    \end{equation*}
\end{cor}

\begin{proof}
    Fix $\gamma > 0$.
    Let $\eps > 0$ be sufficiently small and, in particular, satisfy $(\log(1/\eps))^{-\delta} \leq \gamma$ and $(\frac{3}{2}\varepsilon)^2\le 1/C$ where $\delta > 0$ is the universal constant given by \cref{thm:QCC} and $C>0$ denotes the absolute constant from the inclusion~\eqref{eq:ball}. 
    Define $\eps_0 \coloneqq \eps^2$.
    Let $f \colon (\bH, \dH) \to L_1$ be 1-Lipschitz.
    
    By \cref{thm:QCC}, there is $\rho \geq \eps/2$ and a point $z \in B_{\dH}((0, 0, 0), 1/2)$ such that $\mu(S^{\leq}) \geq \mu(S)/2$ where
    \begin{align*}
        S & \coloneqq \set{(x, y) \in B_{\dH}(z, \rho) \times B_{\dH}(z, \rho) \colon y \in x \cdot Z(\bH), \, \dH(x, y) \in [\tfrac{1}{2}\eps \rho, \tfrac{3}{2} \eps \rho]}, \\
        S^{\leq} & \coloneqq \set{(x, y) \in S \colon \norm{f(x) - f(y)}_{L_1} \leq \gamma \cdot \tfrac{3}{2} \eps \rho}.
    \end{align*}
    Set $r \coloneqq 2 \eps \rho$ and note that $r \geq \eps^2 = \eps_0$.
    We now show that $\mu(S) \gtrsim r^2 \rho^4$. 
    Let $z = (a, b, c)$.
    By \eqref{eq:ball} and \eqref{eq:dHv}, $S$ contains
    \begin{align*}
        \set[\Big]{\bigl((a', b', c'), (a', b', c'')\bigr) \colon & a' \in [a \pm \rho/C], \, b' \in [b \pm \rho/C], \\
        & c', c'' \in [c + ab' - ba' \pm \rho^2/C], \abs{c' - c''} \in \bigl[(\tfrac{1}{4}r)^2, (\tfrac{3}{4} r)^2\bigr]}.
    \end{align*}
    By our initial choice of $\eps > 0$, we have $(\frac{3}{2} \eps)^2 \leq 1/C$ and hence $(\tfrac{3}{4} r)^2 \leq \rho^2/C$. It then follows that the above set has Lebesgue measure at least $(2\rho/C) \times (2\rho/C) \times (\rho^2/C) \times (\tfrac{8}{16} r^2) = 2 r^2 \rho^4/C^3$. Hence, we indeed have $\mu(S) \gtrsim r^2 \rho^4$.

    Define, for $x \in B_{\dH}(z, \rho)$,
    \begin{equation*}
        T_x \coloneqq \set[\big]{y \in B_{\dH}(z, \rho) \colon y \in x \cdot Z(\bH), \, \dH(x, y) \in [\tfrac{1}{4} r, \tfrac{3}{4} r], \, \norm{f(x) - f(y)}_{L_1} \leq \gamma \cdot \tfrac{3}{4} r},
    \end{equation*}
    so $S^{\leq} = \bigcup_{x \in B_{\dH}(z, \rho)} (\set{x} \times T_x)$. Recall that $\mu(S^{\leq}) \geq \mu(S)/2 \gtrsim r^2 \rho^4$ and so
    \begin{align*}
        r^2 \rho^4 \lesssim \mu(S^{\leq}) = \int_{x \in B_{\dH}(z, \rho)}\mu(T_x) \ \textrm{d}\cL_3(x).
    \end{align*}
    Thus, by averaging, there is some $x \in B_{\dH}(z, \rho)$ such that
    \begin{equation*}
        \mu(T_x) \gtrsim \tfrac{r^2 \rho^4}{\mu(B_{\dH}(z, \rho))} \overset{\ref{lem:Haar-ball}}{\asymp} r^2.
    \end{equation*}
    Finally, note that every $y \in T_x$ is in $x \cdot Z(\bH) \cap B_{\dH}(x, r)$ and satisfies $\norm{f(x) - f(y)}_{L_1} \leq \gamma r$ so
    \begin{align*}
        & \frac{\mu(\set{y \in x \cdot Z(\bH) \cap B_{\dH}(x, r) \colon \norm{f(x) - f(y)}_{L_1} \leq \gamma r})}{\mu(x \cdot Z(\bH) \cap B_{\dH}(x, r))} \\
        \geq & \ \frac{\mu(T_x)}{\mu(x \cdot Z(\bH) \cap B_{\dH}(x, r))} \overset{\ref{lem:Haar-coset}}{=} \frac{\mu(T_x)}{2r^2} \gtrsim 1,
    \end{align*}
    as required.
\end{proof}

\section{Discretising}\label{sec:discrete}

The following extension lemma follows from a general result of~\cite{LN05}, which states that such extensions are possible for any Banach-space-valued function from any doubling subset of a metric space to the entire metric space. It also follows from a straightforward partition-of-unity argument.

\begin{lem}\label{lem:extend}
    There is a universal constant $\alpha > 0$ such that if $f \colon (\bH(\bZ), d_T) \to L_1$ is $1$-Lipschitz, then there is $\tilde{f} \colon (\bH, \dH) \to L_1$ which extends $f$ and is $\alpha$-Lipschitz.
\end{lem}

The following lemma gives an important rounding function which bridges between $\bH$ and $\bH(\bZ)$.

\begin{lem}\label{lem:round}
    There is a function $\round \colon \bH \to \bH(\bZ)$ such that
    \begin{itemize}[noitemsep]
        \item $\round$ is the identity on $\bH(\bZ)$,
        \item $\dH(x, \round(x)) \lesssim 1$ for all $x \in \bH$,
        \item if $y \in x \cdot Z(\bH)$, then $\round(y) \in \round(x) \cdot Z(\bH(\bZ))$.
    \end{itemize}
\end{lem}

\begin{proof}
 Note that $\abs{a - \rou{a}} \leq 1/2$ for all $a \in \bR$. Fix $x = (a, b, c) \in \bR^3$. Let $a' = \rou{a}$, $b' = \rou{b}$, and $c' = \rou{c + ab' - ba'}$. Define $\round(x) \coloneqq (a', b', c')$. Note that $\round(x) \in \bZ^3$ and, if $a, b, c \in \bZ$, then $\round(x) = x$. In particular, $\round \colon \bH \to \bH(\bZ)$ is the identity on $\bH(\bZ)$. Now, by \eqref{eq:dH},
    \begin{equation*}
        \dH(x, \round(x)) \lesssim \sqrt{(1/2)^2 + (1/2)^2} + \abs{1/2}^{1/2} = \sqrt{2},
    \end{equation*}
    which gives the second bullet point. Finally, if $y = (a', b', c')$ is in $x \cdot Z(\bH)$ where $x = (a, b, c)$, then $(a', b') = (a, b)$ and so $(\rou{a'}, \rou{b'}) = (\rou{a}, \rou{b})$ and hence $\round(y) \in \round(x) \cdot Z(\bH(\bZ))$.
\end{proof}

The following is our discretised version of the quantitative central collapse theorem.

\begin{lem}\label{lem:discretequantitativecollapse}
    For all sufficiently small $\eps > 0$ and $1$-Lipschitz $f \colon (\bH(\bZ), d_T) \to L_1$, there are arbitrarily large $M$ for which there exists $x \in \bH(\bZ)$ and $K \subset x \cdot Z(\bH(\bZ))$ such that $\abs{K} \geq M^2$ and, for all $y, z \in K$,
    \begin{equation*}
        \norm{f(y) - f(z)}_{L_1} \leq \eps M.
    \end{equation*}
\end{lem}

\begin{proof}
    Let $c_{\round} > 0$ be an absolute constant such that, for the function $\round$ given in \cref{lem:round}, $\dH(x, \round(x)) < c_{\round}$. Let $\alpha$ be the absolute constant given by \cref{lem:extend}.
    
    We will now prove that there is an absolute constant $C > 0$ such that the following holds for all $\eps > 0$. There is $\eps_0 > 0$ such that, for all 1-Lipschitz $f \colon (\bH(\bZ), d_T) \to L_1$ and $R > 0$, there are $r \geq \eps_0$, $\tilde{x} \in \bH(\bZ)$, and $K \subset \tilde{x} \cdot Z(\bH(\bZ))$ of size at least $(rR/C)^2$ such that
    \begin{equation*}
        \norm{f(\tilde{y}) - f(\tilde{z})}_{L_1} \leq 2 \alpha \eps^2 r R + 2 \alpha c_{\round},
    \end{equation*}
    for every $\tilde{y}, \tilde{z} \in K$. We then deduce the lemma statement at the end.

    Fix $\eps > 0$. Let $\eps_0 > 0$ be the constant obtained when \cref{cor:quant-centr-collapse} is applied to $\gamma = \eps^2$. 
    Let $f \colon (\bH(\bZ), d_T) \to L_1$ be 1-Lipschitz and $R > 0$. 
    By \cref{lem:extend} there is a function $\tilde{f} \colon (\bH, \dH) \to L_1$ which extends $f$ and is $\alpha$-Lipschitz. Define
    \begin{equation*}
        f_R \coloneqq \tfrac{1}{\alpha R} \cdot \tilde{f} \circ A_R \colon (\bH, \dH) \to L_1.
    \end{equation*}
    Since $A_R$ is an $R$-homothety of $(\bH, \dH)$, $f_R$ is $1$-Lipschitz. 
    By \cref{cor:quant-centr-collapse}, there is $r \geq \eps_0$ and $x \in \bH$ such that
    \begin{equation*}
        \norm{f_R(x) - f_R(y)}_{L_1} \leq \eps^2 r
    \end{equation*}
    for a constant proportion of $y \in x \cdot Z(\bH) \cap B_{\dH}(x, r)$. Let
    \begin{equation*}
        \cY \coloneqq \set[\big]{y \in x \cdot Z(\bH) \cap B_{\dH}(x, r) \colon \norm{f_R(x) - f_R(y)}_{L_1} \leq \eps^2 r},
    \end{equation*}
    and so
    \begin{equation*}
        \mu(\cY) \gtrsim \mu\bigl(x \cdot Z(\bH) \cap B_{\dH}(x, r)\bigr) \overset{\ref{lem:Haar-coset}}{=} 2r^2.
    \end{equation*}
    Define
    \begin{align*}
        \cY' & \coloneqq \set{A_R y \colon y \in \cY}, \\
        \tilde{x} & \coloneqq \round(A_R x), \\
        K & \coloneqq \set{\round(y') \colon y' \in \cY'}.
    \end{align*}
    Note that $\tilde{x} \in \bH(\bZ)$, by the definition of $\round$. Next, if $\tilde{y} \in K$, then $\tilde{y} = \round(A_R y)$ for $y \in \cY \subset x \cdot Z(\bH)$. Now $A_R y \in A_R x \cdot Z(\bH)$ and so, by \cref{lem:round}, $\tilde{y} \in \tilde{x} \cdot Z(\bH(\bZ))$. Hence $K \subset \tilde{x} \cdot Z(\bH(\bZ))$.

    We next show that $\abs{K} \geq (rR/C)^2$ for an absolute constant $C > 0$. By definition of $K$ and $c_{\round}$,
    \begin{equation*}
        \cY' \subseteq \bigcup_{\tilde{y} \in K} B_{\dH}(\tilde{y}, c_{\round}).
    \end{equation*}
    For each $\tilde{y} \in K$, choose some $y' \in \cY'$ such that $\round(y') = \tilde{y}$. Let $K'$ be the set of chosen $y'$. Then $\abs{K'} \le \abs{K}$ and, by \cref{lem:round}, $\cY' \subset \bigcup_{y' \in K'} B_{\dH}(y', 2c_{\round})$.
    But then, since $A_{1/R}$ is a $1/R$-homothety,
    \begin{equation*}
        \cY = A_{1/R} \cY' \subset \bigcup_{y' \in K'} B_{\dH}(A_{1/R} y', 2c_{\round}/R).
    \end{equation*}
    Now $\cY \subset x \cdot Z(\bH)$ and so $\cY \subset \cup_{y' \in K'} \bigl(B_{\dH}(A_{1/R} y', 2c_{\round}/R) \cap x \cdot Z(\bH)\bigr)$. Note that $A_{1/R} y' \in x \cdot Z(\bH)$ and so, by \cref{lem:Haar-coset},
    \begin{equation*}
        r^2 \lesssim \mu(\cY) \leq \sum_{y' \in K'} \mu\bigl(B_{\dH}(A_{1/R} y', 2c_{\round}/R) \cap x \cdot Z(\bH)\bigr) \leq \abs{K'} \cdot 2(2 c_{\round}/R)^2,
    \end{equation*}
    from which it follows that $\abs{K} \geq \abs{K'} \geq (rR/C)^2$ for an absolute constant $C > 0$. Finally, let $\tilde{y}, \tilde{z} \in K$ and let $y, z \in \cY$ such that $\tilde{y} = \round(A_R y)$ and $\tilde{z} = \round(A_R z)$. By the definition of $\cY$ and the triangle inequality, $\norm{f_R(y) - f_R(z)}_{L_1} \leq \norm{f_R(x) - f_R(y)}_{L_1} + \norm{f_R(x) - f_R(z)}_{L_1} \leq 2 \eps^2 r$. But then, by the definition of $f_R$,
    \begin{equation*}
        \norm{\tilde{f}(A_Ry) - \tilde{f}(A_Rz)}_{L_1} \leq 2 \alpha \eps^2 r R.
    \end{equation*}
    Now $\tilde{f}$ is $\alpha$-Lipschitz and extends $f$ so, by the triangle inequality,
    \begin{align*}
        \norm{f(\tilde{y}) - f(\tilde{z})}_{L_1} & = \norm{\tilde{f}(\tilde{y}) - \tilde{f}(\tilde{z})}_{L_1} \\ 
        & \leq \norm{\tilde{f}(\tilde{y}) - \tilde{f}(A_R y)}_{L_1} + \norm{\tilde{f}(A_R y) - \tilde{f}(A_R z)}_{L_1} + \norm{\tilde{f}(A_R z) - \tilde{f}(\tilde{z})}_{L_1} \\
        & \leq \alpha c_{\round} + 2 \alpha \eps^2 r R + \alpha c_{\round} = 2 \alpha \eps^2 r R + 2 \alpha c_{\round},
    \end{align*}
    as claimed. We now deduce the lemma statement from this.
    Let $\eps > 0$ be sufficiently small so that $4 \alpha C \eps \leq 1$ and let $f \colon (\bH(\bZ), d_T) \to L_1$ be 1-Lipschitz. We have just shown that there is $\eps_0 > 0$ such that for all $R > 0$, there are $r \geq \eps_0$, $\tilde{x} \in \bH(\bZ)$, and $K \subset \tilde{x} \cdot Z(\bH(\bZ))$ of size at least $(rR/C)^2$ such that $\norm{f(\tilde{y}) - f(\tilde{z})}_{L_1} \leq 2 \alpha \eps^2 r R + 2 \alpha c_{\round}$ for every $\tilde{y}, \tilde{z} \in K$.
    
    Let $R$ be sufficiently large so that $R \geq c_{\round} \eps^{-2} \eps_0^{-1}$. Let $r\ge \varepsilon_0$, $\tilde{x}\in \mathbb{H}(\mathbb{Z})$, and $K\subseteq \tilde{x}\cdot Z(\mathbb{H}(\mathbb{Z}))$ be given as in the statement above, and let us define $M \coloneqq rR/C$ and $x \coloneqq \tilde{x}$. Note that $M \geq \eps_0 R/C$, which means we can make $M$ as large as we want by choosing $R$ sufficiently large.
    Then $\abs{K} \geq M^2$ and
    \begin{equation*}
        2 \alpha \eps^2 r R + 2 \alpha c_{\round} \leq 4 \alpha \eps^2 rR \leq \eps \cdot rR/C = \eps M,
    \end{equation*}
    so $x$ and $K$ satisfy the required properties.
\end{proof}

\section{Isometric \texorpdfstring{$L_1$}{l1}-embeddings of infinite trees}\label{sec:isometry}

Recall that a (possibly infinite) graph $T$ is a \defn{tree} if it is connected (i.e., any two vertices can be connected by a finite-length path) and contains no finite cycles. 
In our proof of \cref{thm:main} we will require the following simple lemma about isometric embeddings of countable trees into $L_1$.

\begin{lem}\label{lem:isometrictreeembedding}
    Let $T$ be a finite or countably infinite tree. Then there exists an isometric embedding of $T$ into $L_1$. In other words, there exists a function $\iota:V(T)\rightarrow L_1$ such that for every pair of vertices $u,v\in V(T)$, we have $\dist_T(u,v)=\norm{\iota(u) - \iota(v)}_{L_1}$. 
\end{lem}

\begin{proof}
Since every finite tree is an isometric subgraph of some infinite tree, we may assume that $T$ is infinite. Enumerate the edges of $T$ as $(e_i)_{i \in \bN}$.

Fix some arbitrary vertex $r\in V(T)$, and for each vertex $v \in V(T)$, let $E_v\subseteq E(T)$ denote the (finite) set of edges on the unique path from $v$ to $r$ in $T$. Now, for each $v\in V(T)$, let us define $S_v \coloneqq \bigcup_{i\in \bN \colon e_i\in E_v}[i - 1, i)\subseteq \bR$, and let
$\iota(v) \coloneqq \mathbf{1}_{S_v}\in L_1$ be the indicator function of this set. Then $\norm{\iota(u) - \iota(v)}_{L_1} = \abs{E_u \triangle E_v}$ for every pair $u, v \in V(T)$. Since $T$ is a tree, the symmetric difference $E_u \triangle E_v$ is exactly the set of edges on the unique path connecting $u$ and $v$, and hence $\norm{\iota(u) - \iota(v)}_{L_1} = \dist_T(u, v)$, as desired.
\end{proof}

\section{Proof of \texorpdfstring{\cref{thm:main}}{main theorem}}\label{sec:proof}

In this section, we put all the ingredients developed across the previous sections together to finally deliver the proof of our main result.

\begin{proof}[Proof of \cref{thm:main}]
Fix $C \in \bN$ and suppose, towards a contradiction, that there are finite or countably infinite trees $F_1, F_2, F_3, F_4$ such that $f_{F_i}(r) \leq Cr$, for all $i$ and $r$, and $G \coloneqq \Cay(\bH(\bZ), T)$ is isomorphic to a subgraph of the product $P \coloneqq F_1 \boxtimes F_2 \boxtimes F_3 \boxtimes F_4 \boxtimes K_C$.

Let $\Phi \colon \bH(\bZ) = V(G) \to V(F_1) \times V(F_2) \times V(F_3) \times V(F_4) \times [C]$ be an embedding of $G$ into $P$, and let $\Phi_i \colon V(G) \to V(F_i)$ for $i = 1, \dots, 4$ and $\Phi_5 \colon V(G) \to [C] = V(K_C)$ be the coordinate-restrictions of $\Phi$. Then, by definition of the strong product and since $\Phi$ is an embedding, every two  $u, v \in V(G)$ satisfy
\begin{equation*}
    \max_{1\le i \le 4}\dist_{F_i}(\Phi_i(u),\Phi_i(v))\le \dist_P(\Phi(u),\Phi(v))\le \dist_G(u,v).
\end{equation*}
For each $i \in \set{1, 2, 3, 4}$, fix an isometric embedding $\iota_i$ of $F_i$ into $L_1$ as given by \cref{lem:isometrictreeembedding}.

We further define a linear map $F \colon L_1^4 \to L_1$ as follows. Let $(g_1, g_2, g_3, g_4) \in L_1^4$ and $x \in \bR$. $x$ can be uniquely written as $x = 4\ell + i + z$ for $\ell \in \bZ$, $i \in \set{1, 2, 3, 4}$ and $z\in [0,1)$. We set $F(g_1,g_2,g_3,g_4)(x) \coloneqq g_i(\ell + z)$. 

It is then easy to check that $F$ is well-defined, a bijection, and satisfies $\norm{F(g_1,g_2,g_3,g_4)}_{L_1} \allowbreak = \norm{g_1}_{L_1} + \norm{g_2}_{L_1} + \norm{g_3}_{L_1} + \norm{g_4}_{L_1} \le 4 \max_{1\le i \le 4} \norm{g_i}_{L_1}$ for all $(g_1, g_2, g_3, g_4) \in L_1^4$. 

We now define a map $f \colon \bH(\bZ) = V(G) \to L_1$ as follows. For every vertex $v \in V(G)$, we set
\begin{equation*}
    f(v) \coloneqq F(\iota_1(\Phi_1(v)), \iota_2(\Phi_2(v)), \iota_3(\Phi_3(v)), \iota_4(\Phi_4(v))).
\end{equation*}
Every two elements $u, v \in \bH(\bZ)$ satisfy
\begin{align}
    \norm{f(u)-f(v)}_{L_1} & = \sum_{i=1}^4 \norm{\iota_i(\Phi_i(u)) - \iota_i(\Phi_i(v))}_{L_1} \nonumber \\
    & = \sum_{i=1}^4 \dist_{F_i}(\Phi_i(u), \Phi_i(v)) \label{eq:biLip} \\
    & \le 4\max_{1\le i\le 4} \dist_{F_i}(\Phi_i(u),\Phi_i(v)) \nonumber \\
    &\le 4 \dist_G(u,v) = 4d_T(u,v), \nonumber
\end{align}
where the final inequality used the fact that $\Phi_i \colon V(G) \to V(F_i)$ is $1$-Lipschitz with respect to graph distances in $G$ and $F_i$, respectively.
Thus, the map $f \colon (\bH(\bZ), d_T) \to L_1$ is $4$-Lipschitz.

Let $\eps > 0$ be sufficiently small so that \cref{lem:discretequantitativecollapse} applies and $\eps < \frac{1}{49C^{5/2}}$. 

\Cref{lem:discretequantitativecollapse} applied to $\eps$ and the $1$-Lipschitz map $\frac{1}{4}f \colon \bH(\bZ) \to L_1$ gives $M > \frac{3}{\eps}$, $x \in \bH(\bZ)$, and a subset $K \subseteq x \cdot Z(\bH(\bZ))$ such that $\abs{K} \ge M^2$, and, for all $y, z \in K$, 
\begin{equation*}
    \norm{f(y) - f(z)}_{L_1} \leq 4 \eps M.
\end{equation*}
Equation~\eqref{eq:biLip} then implies that $\dist_{F_i}(\Phi_i(y), \Phi_i(z)) \le 4 \eps M$ for all $y, z \in K$. Fix some $k_0 \in K$.

Now set $d \coloneqq 2 \ceil{\eps M} \in \mathbb{N}$ and consider the following subset of $\bH(\bZ)$:
\begin{equation*}
    K' \coloneqq \set[\big]{k\cdot x_1\cdots x_r \colon k\in K, \ x_1, \dots, x_r \in \set{(1,0,0), (0,1,0)}, \ r \in \set{0, 1, \dots, d}}.
\end{equation*}
Note that every element of $K'$ has distance at most $d$ from some element of $K$ under the word metric $d_T$. Thus, since $\Phi_i \colon V(G) \to V(F_i)$ is $1$-Lipschitz with respect to graph distances, for all $z \in K'$ there is some $k \in K$ such that $\dist_{F_i}(\Phi_i(k), \Phi_i(z)) \le d$. The triangle inequality then gives $\dist_{F_i}(\Phi_i(k_0), \Phi_i(z)) \leq d + 4 \eps M \leq 7 \eps M$. Since $z \in K'$ was arbitrary, $\Phi_i(K')$ is a subset of the ball of radius $\floor{7 \eps M}$ around $\Phi_i(k_0)$ in $F_i$. This implies that $\abs{\Phi_i(K')} \le f_{F_i}(\lfloor 7\eps M\rfloor)\le C\lfloor 7\eps M\rfloor\le 7C\eps M$. This holds for all $i \in \set{1, 2, 3, 4}$. Furthermore, since $\abs{\Phi_5(K')} \le C$,
\begin{equation*}
    \abs{\Phi(K')} \le \prod_{i=1}^{5} \abs{\Phi_i(K')} \le (7C\eps M)^4 \cdot C.
\end{equation*}
Recall that $\Phi$ is an embedding and hence injective, so $\abs{K'} \le 7^4 C^5 \eps^4 M^4$.

We next prove that $\abs{K'} \ge \eps^2 M^4$. Write $K = x \cdot Z$ for some $Z \subseteq Z(\bH(\bZ))$ of size $\abs{K}$.
Define 
\begin{equation*}
    R \coloneqq \set{x_1 \dotsb x_r \colon x_1, \dots, x_r \in \set{(1,0,0), (0,1,0)}, \ r \in \set{0, 1, \dots, d}}.
\end{equation*}
Let $R_2 \subseteq \bZ^2$ denote the projection of $R$ to the first two coordinates. It follows from the definition of the multiplication in the discrete Heisenberg group that $R_2 = \set{(a,b) \in \bZ_{\ge 0}^2 \colon a + b \le d}$, and so $\abs{R_2} \ge (d/2)^2 \ge \eps^2M^2$. Let $t \coloneqq \ceil{\eps^2 M^2}$ and $r_1, \dots, r_t \in R$ be such that the projections of the $r_i$ to their first two coordinates are pairwise distinct. Recall that $Z(\bH(\bZ)) = \set{0} \times \set{0} \times \bZ$ and so multiplying an element of the centre by any element of $\bH(\bZ)$ does not change the first two coordinates of the latter. In particular, the sets $Zr_1$, $Zr_2$, \ldots, $Zr_t$ are pairwise disjoint. Hence
\begin{equation*}
    \abs{K'} = \abs{K \cdot R} = \abs{x \cdot Z \cdot R} = \abs{Z\cdot R} \ge \sum_{i=1}^t\abs{Z r_i} = t \abs{Z} = t \abs{K} \ge \eps^2M^4,
\end{equation*}
as desired. Combining our upper and lower bounds for $\abs{K'}$ gives $\eps^2 M^4 \le 7^4 C^5 \eps^4 M^4$, which rearranges to $\eps\ge \frac{1}{49C^{5/2}}$. This contradicts our choice of $\eps$, as required. 
\end{proof}

\section{Compactness}\label{sec:compactness}

In this section we prove the following compactness result.

\begin{lem}\label{lem:compactness-tree-products}
    Let $G$ be a countable connected graph and let $C \in \bN$. Suppose that for every finite subgraph $H \subseteq G$ there are finite trees $T_1, \dots, T_4$ such that
    \[
        \abs{B_{T_i}(x, r)} \le Cr
    \]
    for every $i \in \set{1, \dots ,4}$, every $x\in V(T_i)$, and every $r \in \bN$, and such that $H$ is isomorphic to a subgraph of
    \[
        T_1 \boxtimes T_2 \boxtimes T_3 \boxtimes T_4 \boxtimes K_C.
    \]
    Then there are finite or countably infinite trees $F_1, \dots, F_4$ such that
    \[
        \abs{B_{F_i}(x,r)} \leq Cr
    \]
    for every $i \in \set{1, \dots ,4}$, every $x\in V(F_i)$, and every integer $r\ge 1$, and such that $G$ is isomorphic to a subgraph of
    \[
        F_1 \boxtimes F_2 \boxtimes F_3 \boxtimes F_4 \boxtimes K_C.
    \]
\end{lem}

\Cref{lem:compactness-tree-products} together with \cref{thm:main} immediately imply the following which shows that the tree product conjecture, \cref{con:wood}, is false.

\begin{thm}\label{thm:main-finite}
    For every $C \in \bN$, there exists a finite graph $H \subseteq \Cay(\bH(\bZ), T)$ with growth function at most $27\cdot r^4$, and such that for every quadruple of trees $T_1, \dots, T_4$ satisfying $f_{T_i}(r)\le Cr$, the graph $H$ is not isomorphic to any subgraph of $T_1\boxtimes T_2\boxtimes T_3\boxtimes T_4\boxtimes K_C$.
\end{thm}

\begin{proof}[Proof of \cref{lem:compactness-tree-products}]
    First choose an increasing exhaustion of $G$ by finite connected subgraphs
    \[
        H_1\subseteq H_2\subseteq \cdots \subseteq G, \qquad \bigcup_{n \ge 1} H_n = G,
    \]
    that is, every vertex and every edge of $G$ belongs to some $H_n$.
    For each $n$, apply the hypothesis of the lemma to $H_n$: there are finite trees \defn{$T^n_1$}, \ldots, \defn{$T^n_4$} satisfying the linear ball-growth bound and, for each $i \in \set{1, \dots, 4}$, an embedding
    \[
        \Phi^n \colon H_n \to T^n_1 \boxtimes \dotsb \boxtimes T^n_4 \boxtimes K_C.
    \]
    Write
    \[
        \Phi^n(v) = \bigl(\phi^n_1(v), \dots , \phi^n_4(v), \chi^n(v)\bigr),
    \]
    where $\phi^n_i(v) \in V(T_i^n)$ and $\chi^n(v) \in [C]$.
    For $u, v \in V(H_n)$ and $i \in \set{1, \dots, 4}$, define
    \[
        \rho_i^n(u,v) \coloneqq \min\set{2, \dist_{T_i^n}(\phi_i^n(u),\phi_i^n(v))} \in \set{0, 1, 2}.
    \]
    Thus, in the $i$th tree coordinate, $\rho_i^n(u,v) = 0$ records equality, $\rho_i^n(u, v) = 1$ records adjacency, and $\rho_i^n(u, v) = 2$ records distance at least $2$. For fixed $m$, we define the \defn{pattern function for $H_m$}, \defn{$\cP_m(\bullet)$}, as follows. For $n \geq m$, $\cP_m(n)$ is the finite pattern consisting of the values
    \[
            \bigl(\chi^n(v) \colon v \in V(H_m)\bigr)
            \quad\text{and}\quad
            \bigl(\rho_i^n(u,v) \colon 1 \leq i \leq 4, \, u, v \in V(H_m)\bigr).
    \]
    Since $\cP_1(\bullet)$ can only take finitely many values, there is an infinite $S_1 \subset \bN$ and a pattern $\cP_1$ such that $\cP_1(n) = \cP_1$ for all $n \in S_1$. Similarly, there is an infinite $S_2 \subset S_1$ and a pattern $\cP_2$ such that $\cP_2(n) = \cP_2$ for all $n \in S_2$. Iterating we obtain infinite sets $\bN \supset S_1 \supset S_2 \supset \dots$ and patterns $\cP_1$, $\cP_2$, \ldots\ such that $\cP_m(n) = \cP_m$ for all $n \in S_m$. $\cP_m$ is the \defn{limit pattern for $H_m$}. For $m' > m$, $S_{m'} \subset S_m$ and so $\cP_m$ is the restriction of $\cP_{m'}$ to $V(H_m)$. Thus, there is a pattern \defn{$\cP$} recording a value $\chi(v) \in [C]$, for each $v \in V(G)$, and a value $\rho_i(u, v) \in \set{0, 1, 2}$, for all $1 \leq i \leq 4$, $u, v \in V(G)$, such that $\cP_m$ is the restriction of $\cP$ to $V(H_m)$ for each $m \in \bN$. $\cP$ is the \defn{limit pattern for $G$}. Let $S \subset \bN$ be an infinite set whose $m$th smallest element is in $S_m$ for all $m$. Note that $\chi(v) = \lim_{n \to \infty, n \in S} \chi^n(v)$ and $\rho_i(u, v) = \lim_{n \to \infty, n \in S} \rho_i^n(u, v)$ for all vertices $u, v \in V(G)$ and $i \in \set{1, \dots, 4}$.
    
    For each $i \in \set{1, \dots, 4}$, we define an equivalence relation \defn{$\sim_i$} on $V(G)$: $u \sim_i v$ if and only if $\rho_i(u, v) = 0$. This relation is plainly reflexive and symmetric. Also, note that if $\rho_i^n(u, v) = \rho_i^n(v, w) = 0$, then $\phi_i^n(u) = \phi_i^n(v) = \phi_i^n(w)$ and so $\rho_i^n(u, w) = 0$. Thus the relation is also transitive.
    Let
    \[
            V(F_i) \coloneqq V(G)/{\sim_i},
    \]
    which is a countable set, since $V(G)$ is countable.
    We write \defn{$[u]_i$} for the $\sim_i$-class of $u$. 
    Note that $\rho_i$ can be well-defined on $V(F_i)$: that is $\rho_i([u]_i, [v]_i)$ is well-defined.
    Indeed, if $[u]_i = [u']_i$ and $[v]_i = [v']_i$, then, for large $n \in S$, $\rho_i^n(u, u') = 0 = \rho_i^n(v, v')$ and so $\phi_i^n(u) = \phi_i^n(u')$ and $\phi_i^n(v) = \phi_i^n(v')$ which implies $\rho_i^n(u, v) = \rho_i^n(u', v')$ and hence $\rho_i(u, v) = \rho_i(u', v')$.
    The edges of $F_i$ are given by joining all pairs of classes $[u]_i$, $[v]_i$ where $\rho_i([u]_i, [v]_i) = 1$.
    
    We now show that $F_i$ is a tree. Fix vertices $[u]_i$ and $[v]_i$. Since $G$ is connected, there is a finite path $u = w_0 \dots w_{\ell} = v$ in $G$.
    For each $t$, $w_t w_{t + 1}$ is an edge of $G$ and so $\rho_i^n(w_t, w_{t + 1}) \leq 1$ for all large $n$.
    This implies that $\rho_i(w_t, w_{t + 1}) \leq 1$ for each $t$ and so vertices $[w_t]_i$ and $[w_{t + 1}]_i$ are either equal or adjacent in $F_i$. Thus, $[w_0]_i \dots [w_\ell]_i$ is a lazy walk\footnote{A \defn{lazy walk} is a walk in which consecutive vertices may be the same.} from $[u]_i$ to $[v]_i$ and so $F_i$ is connected.
    Next, suppose that $F_i$ contains a finite cycle $[u_0]_i [u_1]_i \dots [u_{\ell - 1}]_i [u_0]_i$ of length $\ell \geq 3$. This implies $\rho_i(u_s, u_t) \neq 0$ for all $s \neq t$ and $\rho_i(u_t, u_{t + 1}) = 1$ for all $t$ (indices modulo $\ell$). Thus, for sufficiently large $n \in S$, the $\phi_i^n(u_t)$ are distinct vertices of $T_i^n$ and $\phi_i^n(u_t) \phi_i^n(u_{t + 1}) \in E(T_i^n)$ for all $t$ (indices modulo $\ell$). However, this implies that $\phi_i^n(u_0) \phi_i^n(u_1) \dots \phi_i^n(u_{\ell - 1}) \phi_i^n(u_0)$ is a cycle in the tree $T_i^n$, which is the required contradiction. Thus $F_i$ is indeed a tree.
    
    We now verify the linear ball-growth bound in $F_i$. Let $x$ be a vertex of $F_i$ and fix some $u \in V(G)$ with $x = [u]_i$. Let $[v]_i \in B_{F_i}(x, r)$. Let $[w_0]_i \dots [w_\ell]_i$ (where $w_0 = u$, $w_{\ell} = v$) be a path from $[u]_i$ to $[v]_i$ of length $\ell \leq r$ in $F_i$. For each $t$, $\rho_i(w_t, w_{t + 1}) = 1$ and so, for all large $n \in S$, $\phi_i^n(w_0) \dots \phi_i^n(w_\ell)$ is a walk from $\phi_i^n(u)$ to $\phi_i^n(v)$ in $T_i^n$ of length $\ell$. That is, for all large $n \in S$, $\phi_i^n(v) \in B_{T_i^n}(\phi_i^n(u), r)$. That is, every vertex in $B_{F_i}(x, r)$ is an equivalence class all of whose vertices, $v$, satisfy $\phi_i^n(v) \in B_{T_i^n}(\phi_i^n(u), r)$ for all large $n \in S$. Let $[v_1]_i, \dots, [v_s]_i$ be pairwise distinct vertices of $B_{F_i}(x, r)$. Then, for all large $n \in S$, the $\phi_i^n(v_t)$ are distinct vertices of $B_{T_i^n}(\phi_i^n(u), r)$ and so $s \leq \abs{B_{T_i^n}(\phi_i^n(u), r)} \leq Cr$ by the hypothesis of the lemma. This implies that $\abs{B_{F_i}(x, r)} \leq Cr$, as desired.
    
    We have now constructed the four countable trees $F_1, \dots ,F_4$ with the required linear ball-growth bound.  Define
    \begin{align*}
        \Phi \colon V(G) & \to V(F_1 \boxtimes \dotsb \boxtimes F_4 \boxtimes K_C), \\
        v & \mapsto \bigl([v]_1,[v]_2,[v]_3,[v]_4,\chi(v)\bigr).
    \end{align*}
    It remains to prove that $\Phi$ is an embedding of $G$ into a subgraph of the strong product.
    Let $u, v \in V(G)$ with $\Phi(u) = \Phi(v)$. Then $\chi(u) = \chi(v)$ and $\rho_i(u, v) = 0$ for every $i$. This implies that, for all large $n \in S$, $\chi^n(u) = \chi^n(v)$ and $\phi_i^n(u) = \phi_i^n(v)$. This implies that $\Phi^n(u) = \Phi^n(v)$. However $\Phi^n$ is an embedding and so injective which implies $u = v$. Thus $\Phi$ is injective.
    
    Finally, let $uv \in E(G)$. Then $\rho_i^n(u,v) \leq 1$ for all large $n$ and so $\rho_i(u, v) \leq 1$ which implies that vertices $[u]_i, [v]_i$ are adjacent or equal in $F_i$ for all $i \in \set{1, \dots, 4}$. This implies that vertices $\Phi(u)$ and $\Phi(v)$ are equal or adjacent in $F_1 \boxtimes \dots \boxtimes F_4 \boxtimes K_C$ (since $K_C$ is complete the values of $\chi(u), \chi(v)$ do not matter). By injectivity of $\Phi$, they must in fact be adjacent, as required.
\end{proof}

{
\fontsize{11pt}{12pt}
\selectfont
	
\hypersetup{linkcolor={red!70!black}}
\setlength{\parskip}{2pt plus 0.3ex minus 0.3ex}

\newcommand{\etalchar}[1]{$^{#1}$}

}

\end{document}